\documentclass[12pt]{amsart}

\usepackage{amscd,amssymb,times, amsmath,color, fullpage}
\usepackage{enumitem}
\newtheorem{theorem}{Theorem}

\newtheorem{proposition}[theorem]{Proposition}

\theoremstyle{definition}
\newtheorem{definition}[theorem]{Definition}

\newtheorem{example}[theorem]{Example}

\theoremstyle{remark}

\begin{document}

\title{On characteristic foliations of metric contact-symplectic structures}
\author{Amine Hadjar}
\address{%
D\'epartement de Math\'ematiques, IRIMAS, Universit{\'e} de Haute Alsace
\newline
18, Rue des Fr{\`e}res Lumi{\`e}re, 68093 Mulhouse Cedex, France}
\email{mohamed.hadjar@uha.fr}

\date{\today; 2020 Mathematics Subject Classification: primary 53C25; secondary 53B20, 53D10, 53B35, 53C12}

\maketitle
\vspace{5mm}

\begin{abstract}
We study compatible and associated metrics for a contact-symplectic pair $(\eta , \omega)$ on a manifold.
We show that the integral curves of the Reeb vector field are geodesics for any compatible metric. 
We prove that all associated metrics share a common volume element, which we give explicitly.
When the characteristic foliations of $\eta$ and $\omega$ are orthogonal   with respect to an associated metric, their leaves, as well as those of the characteristic foliation of $d\eta$, are minimal. 
We construct explicit examples on nilpotent Lie groups and nilmanifolds where the characteristic foliations are not both totally geodesic.
\end{abstract}

\section{Introduction} 
Bande \cite{B}  introduced the notion of \emph{contact-symplectic pairs},
which he later studied together with Ghiggini and Kotschick \cite{  BGK, BKsp}.
Subsequently, he and the present author  adapted associated metrics to this structure in order to construct contact pair structures on Boothby--Wang fibrations \cite{BHnormal}.

\medskip 
\noindent
A contact-symplectic pair on a manifold $M$ consists of a $1$-form $\eta$ and a $2$-form $\omega$ such that the triple $(M,\eta,\omega)$ is locally the product of a contact manifold and a symplectic manifold.
These manifolds are naturally foliated by the characteristic foliation $\mathcal{S}$ of $\eta$ and the characteristic foliation $\mathcal{C}$ of $\omega$, which are transverse and complementary.
The form $\eta$ induces contact forms on the leaves of $\mathcal{C}$, while $\omega$ induces symplectic forms on the leaves of $\mathcal{S}$.
The Reeb vector fields on the contact leaves together define a global vector field $\xi$, called the Reeb vector field of the contact-symplectic pair.

\medskip 
\noindent
Compatible and associated metrics are defined analogously to those in contact geometry, beginning with a $(1,1)$-tensor field $\phi$ 
that vanishes on $\xi$ and  restricts to an almost complex structure on the kernel of $\eta$.
Thus, the triple $(\eta,  \xi, \phi)$ determines an almost contact structure, with  $\omega$ omitted. 
For an associated metric $g$, the triple $(\eta,\omega,g)$ is called a \emph{metric contact-symplectic structure} on $M$.

\medskip 
\noindent
In this paper, we prove that the integral curves of the Reeb vector field of a contact-symplectic pair are geodesics with respect to any compatible metric.
We also show that all associated metrics share the same volume element, and we provide an explicit expression for it in terms of the pair $(\eta,\omega)$.
For associated metrics for which  the foliations $\mathcal{C}$ and $\mathcal{S}$ are orthogonal, we prove that all their leaves are minimal. 
In addition, we show that the leaves of the characteristic foliation of $d\eta$ are also minimal. 

\medskip 
\noindent
We provide examples of nilpotent Lie groups and nilmanifolds endowed with  metric contact-symplectic structures whose characteristic foliations are orthogonal but not both totally geodesic.

\medskip 
\noindent
All differential objects considered are assumed to be smooth.

\section{Preliminaries}

A pair $(\eta , \omega)$ consisting of a $1$-form $\eta$ and a closed $2$-form $\omega$ on a manifold $M$ is called a \emph{contact-symplectic pair of type $(m,n)$} if the following conditions are satisfied:
\begin{align*}
&\eta\wedge (d\eta)^{m}\wedge\omega^{n} \quad \text{is a volume form},\\
&(d\eta)^{m+1}=0 \quad  \text{and} \quad \omega^{n+1}=0.
\end{align*}
Such a manifold is necessarily orientable and of odd dimension $2n+2m+1$.
Since the cases $n=0$ or $m=0$ correspond to classical contact and cosymplectic manifolds, respectively, we will henceforth assume $m\geq 1$ and $n\geq 1$.

\medskip 
\noindent
The forms $\eta$ and  $\omega$ have constant \'Elie Cartan classes, 
equal to $2m+1$ and $2n$, respectively.
As a consequence, the distributions
\begin{align*}
&T\mathcal{S}=\{X\in TM \mid \eta(X)=0, d\eta(X,Y)=0 \;  \forall Y \in TM  \} \\
&T\mathcal{C}=\{X \in TM  \mid \omega(X,Y)=0 \;  \forall Y \in TM  \}
\end{align*}
are integrable and determine two transverse and complementary foliations on $M$, denoted by $\mathcal{S}$ and $\mathcal{C}$, respectively.
The leaves of $\mathcal{S}$ are symplectic manifolds of dimension $2n$, with symplectic form induced by $\omega$, while the leaves of $\mathcal{C}$ are contact manifolds of dimension $2m+1$, with contact form induced by $\eta$ (see \cite{B}).

\medskip 
\noindent
The vector field $\xi$ defined uniquely by the equations
\begin{equation*}
\eta(\xi) = 1, \quad
i_\xi d\eta = 0, \quad
i_\xi \omega = 0
\end{equation*}
is called the \emph{Reeb vector field} of the contact-symplectic pair.
It is tangent to the foliation $\mathcal{C}$ and restricts to the classical Reeb vector field on each contact leaf (see \cite{B}).
Its flow preserves both forms $\eta$ and $\omega$:
\begin{equation*}
\mathcal{L}_\xi \eta = 0, \quad \mathcal{L}_\xi \omega = 0,
\end{equation*}
where $\mathcal{L}_X$ denotes the Lie derivative along a vector field $X$.

\medskip 
\noindent
The characteristic subbundle of  $d\eta$ is $\mathbb{R}\xi \oplus T\mathcal{S}$, which is integrable. The corresponding foliation will be denoted by $\mathcal{K}$.
Each of its leaves is a $(2n+1)$-dimensional cosymplectic manifold foliated by leaves of $\mathcal{S}$ (see \cite{B}).

\medskip 
\noindent
We also define the  \emph{horizontal subbundle} $ \mathcal{H}$, of dimension $2n+2m$, to be the kernel of $\eta$, 
and the  subbundle $\mathcal{D}=\mathcal{H}\cap T\mathcal{C}$, of dimension $2m$, which corresponds to the contact distribution along the leaves of $\mathcal{C}$.
The tangent bundle of $M$ thus splits as:
\begin{equation*}
TM=T\mathcal{S}\oplus T\mathcal{C} = T\mathcal{S}\oplus \mathbb{R}\xi  \oplus \mathcal{D} = T\mathcal{K} \oplus  \mathcal{D}=  \mathcal{H}\oplus \mathbb{R}\xi .
\end{equation*}

\noindent
A basic example is given by the product of a $(2m+1)$-dimensional contact manifold with a $2n$-dimensional symplectic manifold.
In fact, any contact-symplectic pair of type $(m,n)$ is locally isomorphic to such a product. 

\medskip 
\noindent
 An \emph{almost contact-symplectic structure} on a manifold $M$ is a triple $(\eta, \omega, \phi)$ where $(\eta, \omega)$ is a contact-symplectic pair with Reeb vector field $\xi$, and $\phi$ is a $(1,1)$-tensor field satisfying
\begin{equation*}
 \phi^2=-I +\eta \otimes \xi.
\end{equation*}
This condition ensures that the triple $(  \phi , \xi,   \eta)$ defines an almost contact structure on $M$, and in particular: 
\begin{equation*}
\phi \xi=0 , \quad \eta \circ \phi =0
\end{equation*}
and $\phi$ induces an almost complex structure on the horizontal subbundle $\mathcal{H}$.
 
\medskip 
\noindent
On manifolds endowed with almost contact-symplectic structures, one can define Riemannian metrics adapted to the structure, as follows:

\begin{definition}
Let $(\eta, \omega, \phi)$ be an almost  contact-symplectic structure on a manifold $M$ with Reeb vector field $\xi$.
A Riemannian metric $g$ on $M$ is said to be:
\begin{enumerate}[label=(\roman*)]
\item \emph{compatible}  if, for all vector fields $X$, $Y$,
	\begin{equation*}
	g(\phi X,\phi Y)=g(X,Y)-\eta (X)\eta (Y)
	\end{equation*}
\item \emph{associated}  if, for all vector fields $X$, $Y$,
	 \begin{align*}
	&g(X,\phi Y)=(\mathrm{d}\eta +\omega)(X,Y) \\
	&g(X,\xi)=\eta (X)
	\end{align*} 
\end{enumerate}
\end{definition}

 \noindent
 A metric $g$ is compatible with the almost contact-symplectic structure $(\eta, \omega, \phi)$ if and only if the quadruple $( \phi, \xi, \eta, g)$ is an almost contact metric structure (see \cite{Blairbook} for a definition).
So an almost c{proposition}ontact-symplectic pair structure $(\eta, \omega, \phi)$ always admits a compatible metric. 
In particular, for any compatible metric $g$, we have 
 \begin{equation*}
g(\xi,\xi)=1\quad  \text{and} \quad g(X,\xi)=\eta(X),
\end{equation*}
so that the subbundles $\mathcal{H}$ and $\mathbb{R}\xi$ are orthogonal, $\xi$ being the Reeb vector field of the contact-symplectic pair.

 \begin{proposition}
For an almost contact-symplectic structure, every associated metric is compatible.
However, the converse does not hold in general.
 \end{proposition}
 \begin{proof}
Let $g$ be an associated metric. Then:
 \begin{align*}
g(\phi X,\phi Y)&=(\mathrm{d}\eta +\omega)(\phi X,Y) \\
			&=-(\mathrm{d}\eta +\omega)(Y, \phi X) \\
			&=-g(Y,\phi^2X) \\
			&=-g(Y,-X+\eta(X)\xi)\\
			&=g(X,Y)-\eta(X)\eta(Y),
\end{align*}
which shows that $g$ is compatible.

\medskip 
\noindent
To prove that the converse fails in general, suppose that  $(\eta, \omega, \phi)$  is an almost contact-symplectic structure.
Then $(\eta, - \omega, \phi)$ also defines such a structure.
Take a metric $g$ compatible with the common almost contact structure $( \phi, \xi, \eta)$. It cannot be simultaneously associated to both, since this would imply
 \begin{equation*}
(\mathrm{d}\eta+\omega)(X,Y)=g(X,\phi Y)=(\mathrm{d}\eta-\omega)(X,Y), 
\end{equation*}
which leads to 
 \begin{equation*}
\omega(X,Y) = 0
 \end{equation*}
 for all $X, Y$, contradicting the fact that $\omega$ is symplectic on $T\mathcal{S}$.
 \end{proof}

\noindent
A \emph{metric contact-symplectic structure} is a 4-tuple $(\eta, \omega, \phi, g)$ where the triple $(\eta, \omega, \phi)$ is an almost contact-symplectic structure and $g$ a Riemannian metric associated to it.
Note that  the required equations use the objects $\xi$ and $\phi$ that are uniquely defined by $(\eta, \omega, g)$.
Hence, the definition can equivalently be reformulated as follows:

\begin{definition}
A \emph{metric contact-symplectic structure}  on a manifold is a triple $(\eta, \omega, g)$ where $(\eta, \omega)$ is a contact-symplectic pair with Reeb vector field $\xi$, and $g$ a Riemannian metric such that:
\begin{enumerate}[label=(\roman*)]
\item $g(X,\xi)=\eta (X)$, and
\item the endomorphism field $\phi$ defined by $g(X,\phi Y)=(\mathrm{d}\eta +\omega)(X,Y)$ satisfies $\phi^2=-I +\eta \otimes \xi$.
\end{enumerate}
\end{definition}

\section{Examples}
\noindent  
We begin with the following straightforward construction.

\begin{example}
The product of a symplectic manifold endowed with an associated metric and a contact metric manifold (see \cite{Blairbook} for definitions) yields a manifold with a metric contact-symplectic structure. 
In this case, the two characteristic foliations are orthogonal and, moreover, both totally geodesic.
\end{example}

\noindent
It is important to note that although any contact-symplectic manifold is locally a product of a contact manifold and a symplectic manifold, there exist metric contact-symplectic manifolds with orthogonal characteristic foliations that are not locally such products.
The following examples illustrate this well.

\begin{example}
Let $G^5$ be the simply connected $5$-dimensional nilpotent Lie group defined by the structure equations
 \begin{equation*}
 \mathrm{d}\alpha_2=\alpha_1\wedge \alpha_3, \quad  \mathrm{d}\alpha_4=\alpha_1\wedge \alpha_5, \quad
 \mathrm{d}\alpha_1= \mathrm{d}\alpha_3= \mathrm{d}\alpha_5=0.
  \end{equation*}
The pair  $(\eta, \omega)$, where $\eta=\alpha_2$ and $\omega=\alpha_4\wedge \alpha_5$, is a contact-symplectic pair of type $(1,1)$ on $G^5$.  
The Reeb vector field is $\xi=X_2$, where the $X_i$'s are dual to the $\alpha_i$'s.
The characteristic distribution $T\mathcal{S}$ of $\eta$ is spanned by $X_4$ and $X_5$, while the characteristic distribution $T\mathcal{C}$  of $\omega$ is spanned by   $X_1$, $X_2$ and $X_3$.
Consider the metric
\begin{equation*}
g=\alpha_2^2+\frac{1}{2}\left( \alpha_1^2+\alpha_3^2+\alpha_4^2+\alpha_5^2 \right)
 \end{equation*}
which becomes an associated metric for the contact-symplectic pair when equipped with the endomorphism field $\phi$ defined by
\begin{equation*}
\phi X_2=0 ,\quad  \phi X_ 3=X_1 ,\quad \phi X_1 =-X_3 ,\quad \phi X_ 5=X_4,\quad \phi X_4 =-X_5 .
 \end{equation*}
With respect to this metric, the two characteristic distributions $T\mathcal{S}$ and $T\mathcal{C}$ are orthogonal.
The foliation $\mathcal{C}$ is easily seen to be totally geodesic.
However, denoting by $\nabla$ the Levi-Civita connection, we have:
\begin{align*}
 	  -2 g(\nabla_{X_4}X_5, X_1 )&=  g([X_5,X_1],X_4) + g([X_4,X_1],X_5)  +g([X_5,X_4],X_1)  \\
	                                              &= g(X_4,X_4)  \\
	                                              &\neq 0
  \end{align*}
This shows that the foliation $\mathcal{S}$ is not totally geodesic.

\medskip 
\noindent
Since the structure constants of the Lie algebra of $G^5$ are rational, the group admits cocompact lattices $\Gamma$. The left-invariant metric contact-symplectic structure $(\eta, \omega, g)$ then descends to any compact nilmanifold $G^5/\Gamma$, yielding a metric contact-symplectic structure of type $(1,1)$ with orthogonal characteristic foliations, where only one of the two is totally geodesic.

\end{example}

\begin{example}
Let $G^7$ be the simply connected $7$-dimensional nilpotent Lie group defined by the structure equations:
 \begin{equation*}
 \mathrm{d}\alpha_{j-1}=\alpha_1\wedge \alpha_j \quad  \text{for} \quad  j=3,4,5,7\quad \text{and}\quad
 \mathrm{d}\alpha_1= \mathrm{d}\alpha_5= \mathrm{d}\alpha_7=0
  \end{equation*}
Set $\eta=\alpha_2$ and $\omega= \alpha_4\wedge \alpha_5+\alpha_6\wedge \alpha_7$. 
Then  $(\eta, \omega)$  is a contact-symplectic pair of type $(1,2)$ on  $G^7$,  with Reeb vector field $\xi=X_2$,  where the $X_i$'s are  dual to the $\alpha_i$'s.
The characteristic distribution $T\mathcal{S}$ of $\eta$  is spanned by $X_4$, $X_5$, $X_6$ and $X_7$, while  the characteristic distribution  $T\mathcal{C}$ of  $\omega$ is spanned by $X_1$, $X_2$ and $X_3$.
Consider the metric
 \begin{equation*}
g=\alpha_2^2+\frac{1}{2}\left( \alpha_1^2+\alpha_3^2+\alpha_4^2+\alpha_5^2 +\alpha_6^2+\alpha_7^2\right)
 \end{equation*}
which is associated to the contact-symplectic pair via the endomorphism field $\phi$ defined to be zero on $X_2$ and:
\begin{align*}
\phi X_ 3&=X_1 ,\quad \phi X_1 =-X_3 , \quad \phi X_ 5=X_4,\\
\phi X_4 &=-X_5 , \quad \phi X_ 7=X_6,\quad \phi X_6 =-X_7.
  \end{align*}
 The characteristic distributions $T\mathcal{S}$ and $T\mathcal{C}$ are again orthogonal with respect to this metric.
We now show that neither foliation is totally geodesic.
Letting $\nabla$ denote the Levi-Civita connection, we compute:
\begin{align*}
 	  -2 g(\nabla_{X_3}X_1, X_4 )&=   g([X_1,X_4],X_3) + g([X_3,X_4],X_1)+g([X_1,X_3],X_4)   \\
	                                              &= -g(X_3,X_3) \\
	                                              &\neq 0
  \end{align*}
  and
  \begin{align*}
 	  -2 g(\nabla_{X_4}X_5, X_1)&=   g([X_5,X_1],X_4)  +g([X_4,X_1],X_5) +  g([X_5,X_4],X_1)   \\
	                                              &= g(X_4,X_4)  \\
	                                              &\neq 0,                       
  \end{align*}
thus proving that both foliations $\mathcal{C}$ and $\mathcal{S}$ fail to be totally geodesic.

\medskip 
\noindent
As in the previous case, the rationality of the structure constants guarantees the existence of cocompact lattices $\Gamma$ in $G^7$. The left-invariant metric contact-symplectic structure $(\eta, \omega, g)$ then descends to each compact nilmanifold $G^7/\Gamma$, yielding a metric contact-symplectic structure of type $(1,2)$ with orthogonal characteristic foliations, neither of which is totally geodesic.

\end{example}

\section{Minimal Foliations}
\noindent 
We first turn our attention to the integral curves of the Reeb vector field of an almost contact-symplectic structure endowed with a compatible Riemannian metric.

\begin{theorem}
Let $(\eta, \omega, \phi, g)$ be an almost contact-symplectic structure on a manifold, with a compatible Riemannian metric $g$.
Then  the integral curves of its Reeb vector field $\xi$ are geodesics.
\end{theorem}
\begin{proof}
Let $\nabla$ denote the Levi-Civita connection of $g$.
Since $\eta$ is invariant under the flow of $\xi$ i.e. $\mathcal{L}_\xi \eta=0$, it follows that for all vector field $X$:
\begin{align*}
 	0&=(\mathcal{L}_\xi \eta ) (X)=\xi (\eta(X))-\eta (\mathcal{L}_\xi X) \\
	  &= \xi g(X,\xi)-g(\xi, \nabla_\xi X - \nabla_X \xi) \\
	  &= g(\nabla_\xi X, \xi)+g(X,\nabla_\xi \xi)-g(\xi, \nabla_\xi X - \nabla_X \xi)\\
	  &= g(X,\nabla_\xi \xi)+g(\xi, \nabla_X \xi)
  \end{align*}
But since $g(\xi, \xi) = 1$, we differentiate both sides to obtain: 
\begin{equation*}
0=Xg(\xi,\xi)=2g(\xi, \nabla_X \xi)
\end{equation*}
Hence, from the earlier computation, $g(X, \nabla_\xi \xi) = 0$ for all $X$, and thus $\nabla_\xi \xi = 0$. Therefore, the integral curves of $\xi$ are geodesics.
\end{proof}

\noindent
Almost contact-symplectic and contact structures are special cases of the broader class of almost contact structures.
The following result provides an explicit expression for the volume element of a Riemannian metric associated with a contact-symplectic pair, generalizing a classical result from contact and symplectic geometry (see \cite{Blairbook}).

\begin{theorem}\label{theoremdV}
For a contact-symplectic pair $(\eta, \omega)$ of type $(m,n)$ on a manifold, all associated metrics have the same volume element given by
\begin{equation*}
dV= \dfrac{(-1)^{m+n}}{2^{m+n} m! \ n!}\ \eta\wedge (d\eta)^{m}\wedge\omega^{n}
\end{equation*}
\end{theorem}
\begin{proof}
We know that for a metric contact-symplectic structure $(\eta, \omega, g, \phi)$ of type $(m,n)$ with Reeb vector field $\xi$ on a manifold $M$, $(\phi, \xi, \eta, g)$ is an almost contact metric structure on $M$.
Consequently, one can locally construct a $\phi$-basis, that is, an orthonormal frame $\{ \xi, X_1, \dots , X_{m+n},  X_{1^*},$ $\dots , X_{(m+n)^*}\}$ where $X_{i^*}=\phi X_i$.
Let $\{ \eta, \theta^1, \dots , \theta^{m+n}, \theta^{1^*}, \dots , \theta^{(m+n)^*}\}$ be its dual basis. 
Then
\begin{equation*}
dV= \eta \wedge \theta^1 \wedge  \theta^{1^*} \wedge  \dots \wedge \theta^{m+n} \wedge \theta^{(m+n)^*}.
\end{equation*}
Observe that
\begin{equation*}
d\eta+\omega =\sum_{i=1}^{m+n} (\theta^{i^*} \wedge \theta^i-\theta^i \wedge  \theta^{i^*})=-2\sum_{i=1}^{m+n}\theta^i \wedge  \theta^{i^*}
\end{equation*}
which follows from the identity $g(X, \phi Y) = (\mathrm{d}\eta + \omega)(X, Y)$ and the orthonormality of the local $\phi$-basis. 
Next we have
\begin{equation*}
(d\eta+\omega )^{m+n}=(-2)^{m+n} (m+n)! \ \theta^1 \wedge  \theta^{1^*} \wedge  \dots \wedge \theta^{m+n} \wedge \theta^{(m+n)^*}.
\end{equation*}
On the other hand,
\begin{align*}
(d\eta+\omega )^{m+n}&=\sum_{k=0}^{m+n} \binom{m+n}{k} (d\eta)^k \wedge \omega^{m+n-k} \\
				   &= \dfrac{(m+n)!}{m!\ n!}(d\eta)^m \wedge \omega^n .
\end{align*}
Comparing both expressions yields:
\begin{align*}
\eta\wedge(d\eta)^m \wedge \omega^n&=(-2)^{m+n} m!\ n! \ \eta \wedge \theta^1 \wedge  \theta^{1^*} \wedge  \dots \wedge \theta^{m+n} \wedge \theta^{(m+n)^*}\\
								&=(-2)^{m+n} m!\ n!  \ dV .
\end{align*}
from which the formula follows.
\end{proof}

\noindent
We now turn to the study of the characteristic foliations arising from a metric contact-symplectic structure 
$(\eta, \omega, g)$. 
Let $\mathcal{S}$, $\mathcal{C}$, and $\mathcal{K}$ denote respectively the characteristic foliations of $\eta$, $\omega$, and $d\eta$. Their leaves are, respectively symplectic, contact and cosymplectic manifolds.

\begin{theorem}
Let $(\eta, \omega, g)$ be a metric contact-symplectic structure on a manifold $M$, and suppose that the characteristic foliations $\mathcal{S}$ of $\eta$ and $\mathcal{C}$ of $\omega$ are orthogonal. Then:
\begin{itemize}[label=--]
\item the leaves of $\mathcal{S}$,
\item the leaves of $\mathcal{C}$,
\item and the leaves of $\mathcal{K}$, the characteristic foliation of $d\eta$,
\end{itemize}
are all minimal submanifolds of the Riemannian manifold $(M, g)$.
\end{theorem}

\begin{proof}
We apply Rummler's minimality criterion (see \cite{Rummler}), which states that a $p$-dimensional foliation $\mathcal{F}$ on a Riemannian manifold is minimal if and only if its characteristic form $\alpha$ is relatively closed on $T\mathcal{F}$; that is,
\begin{equation*}
d\alpha(X_1, \dots, X_p, Y) = 0
\end{equation*}
for all $X_1, \dots, X_p$  tangent to $\mathcal F$ and $Y$ tangent to the manifold.
We recall that $\alpha$ is the $p$-form
which vanishes on vectors orthogonal to $\mathcal F$ and whose restriction to $\mathcal F$ is the volume form of the induced metric on the leaves.

\medskip 
\noindent
Let $(m,n)$ be the type numbers of the contact-symplectic pair $(\eta, \omega)$, and let $\xi$ denote its Reeb vector field.

\medskip 
\noindent
-- The foliation $\mathcal{S}$ has dimension $2n$. Since $T\mathcal{C}$ is the orthogonal subbundle of $T\mathcal{S}$, the characteristic form of $\mathcal{S}$ is :
\begin{equation*}
\alpha_{\mathcal{S}} = \dfrac{(-1)^n}{2^n \, n!} \, \omega^n,
\end{equation*}
which is closed, and thus relatively closed on $T\mathcal{S}$. Therefore, the leaves of $\mathcal{S}$ are minimal.

\medskip 
\noindent
-- The foliation $\mathcal{C}$ has dimension $2m+1$, and its characteristic form is:
\begin{equation*}
\alpha_{\mathcal{C}} = \dfrac{(-1)^m}{2^m \, m!} \, \eta \wedge (d\eta)^m,
\end{equation*}
which is again closed. Hence, the leaves of $\mathcal{C}$ are minimal.

\medskip 
\noindent
-- The foliation $\mathcal{K}$ has dimension $2n+1$, and its characteristic form is:
\begin{equation*}
\alpha_{\mathcal{K}} =  \dfrac{(-1)^n}{2^n \, n!} \, \eta \wedge \omega^n.
\end{equation*}
As $\xi$ is tangent to $\mathcal{K}$ and since
\begin{equation*}
i_\xi d(\eta \wedge \omega^n) = i_\xi (d\eta \wedge \omega^n) = 0,
\end{equation*}
the form $\alpha_{\mathcal{K}}$ is relatively closed on $T\mathcal{K}$. Thus, the leaves of $\mathcal{K}$ are also minimal.
\end{proof}


\begin{thebibliography}{1}


\bibitem{B}
Bande, Gianluca. Couples contacto-symplectiques.  \textit{Transactions of the American Mathematical Society}, 355 (2003), 1699--1711 

\bibitem{BGK}
Bande, Gianluca; Ghiggini, Paolo; Kotschick, Dieter. Stability theorems for symplectic and contact pairs. \textit{International Mathematics Research Notices}, 2004:68 (2004), 3673--3688


\bibitem{BHnormal}
Bande, Gianluca; Hadjar, Amine. On normal contact pairs. \textit{International Journal of Mathematics}, 21 (2010), 737--754 

\bibitem{BKsp}
Bande, Gianluca, Kotschick, Dieter. The geometry of symplectic pairs. \textit{Transactions of the American Mathematical Society}, 358 (2006),1643--1655


\bibitem{Blairbook}
Blair, David: \textit{Riemannian geometry of contact and symplectic manifolds.} Progress in Mathematics, 203 (Second edition). Birkh\"auser Boston (2010)

\bibitem {Rummler}
Rummler, Hansklaus. Quelques notions simples en g\'eom\'etrie riemannienne et leurs applications aux feuilletages compacts. \textit{Commentarii Mathematici Helvetici}, 54 (1979), 224--239

\end{thebibliography}
\end{document}